\documentclass{amsart}   	
\usepackage{geometry}                		
\usepackage{graphicx}				
\usepackage{amsmath}								
\usepackage{amssymb}
\usepackage{amsthm}
\usepackage{fancyhdr}

\newtheorem{theorem}{Theorem}[section]

\newtheorem{lemma}[theorem]{Lemma}

\title{Remarks on shrinking gradient K\"ahler-Ricci solitons with positive bisectional curvature}
\author{Guoqiang Wu}
\address{Department of Mathematics, East China Normal University}
\email{gqwu@math.ecnu.edu.cn}
\author{Shijin Zhang}
\address{School of Mathematics and Systems Science, Beihang University, Beijing 100191, P.R. China}
\email{shijinzhang@buaa.edu.cn}
\date{}							

\begin{document}
\maketitle
\begin{abstract}
\noindent
In this short note, using an argument by Munteanu and Wang, we provide an alternative proof of the fact first obtained by Lei Ni that shrinking gradient K\"ahler-Ricci solitons with positive bisectional curvature must be compact.\end{abstract}

\section{Introduction}
A gradient Ricci soliton is a self-similar solution to the Ricci
flow which flows by diffeomorphism  and homothety.  The study of solitons has become
increasingly important in both the study of the Ricci
flow and metric measure theory. In Perelman's proof of Poincar\'e conjecture , one  issue he needs to prove is that three dimensional shrinking  gradient Ricci soliton with positive sectional curvature is compact. It is natural to ask whether this   holds in high dimension. Last year, Munteaun-Wang \cite{Munteanu-Wang} proved that this is true.
\begin{theorem}[Munteanu-Wang \cite{Munteanu-Wang}]
Any shrinking gradient Ricci soliton with nonnegative sectional curvature and positive Ricci curvature is compact.
\end{theorem}
In this paper, we consider the gradient K\"ahler-Ricci soliton, namely a triple $(M^n, g, f)$ associated with a K\"ahler manifold $(M,g)$ such that
\begin{equation}\label{solitonequation}
R_{i\overline{j}}+\nabla_{i}\nabla_{\overline{j}}f=\beta g_{i\overline{j}},\quad\quad {\rm and}\quad\quad \nabla_{i}\nabla_{j}f=0.
\end{equation}
for some constant $\beta\in {\mathbb R}$. It is called shrinking, steady or expanding, if $\beta>0$, $\beta=0$ or $\beta<0$ respectively.
Using similar method as in \cite{Munteanu-Wang}, we prove the following theorem.
\begin{theorem}\label{Main Theorem}
Shrinking gradient K\"ahler-Ricci solitons with nonnegative bisectional curvature and positive Ricci curvature are compact.
\end{theorem}
Actually, L. Ni \cite{Ni} proved the above theorem using different  method. He firstly proved that the scalar curvature has a uniform positive lower bound, and used the result that the average of scalar curvature on a noncompact K\"ahler manifold with positive bisectional curvature has a decay estimate of the form $\frac{C}{1+r}$; in this way he \cite{Ni}  proved  the following theorem.
\begin{theorem}[Ni \cite{Ni}]\label{Lei Ni}
Let $(M^n, g)$ be a shrinking gradient K\"ahler-Ricci soliton with nonnegative bisectional curvature, then we have\\
{\rm(i)} If the bisectional curvature of $M$ is positive then $M$ must be compact and isometric-biholomorphic to ${\mathbb C}P^{n}$;\\
{\rm(ii)} If $M$ has nonnegative bisectional curvature then the universal cover $\tilde{M}$ splits as $\tilde{M}=N_{1}\times N_{2}\times\cdots\times N_{l}\times {\mathbb C}^{k}$ isometric-biholomorphically, where $N_{i}$ are compact irreducible Hermitian Symmetric Spaces.
\end{theorem}


Next section we will give the proof the Theorem \ref{Main Theorem}. The idea follows from Munteanu-Wang \cite{Munteanu-Wang}: we assume the manifold is noncompact, then we have the asymptotic behavior of potential $f$, the key  is to get some lower bound of Ricci curvature. For this, we introduce a quantity  involving Ricci curvature and potential $f$ as in \cite{Chow-Lu-Yang} and \cite{Munteanu-Wang}, then using maximum principle, we can get the Ricci curvature lower bound which
is related to $f$. Once we have this,  it is easy to obtain that the average of scalar curvature on geodesic ball could be sufficiently large, however, it is impossible, so $M$ must be compact.

\section{Proof of  Theorem \ref{Main Theorem}}
For simplicity, we assume $\beta=1$. Firstly, we obtain a few formulas for gradient K\"ahler-Ricci solitons.
\begin{lemma}\label{several Identity}
{\rm(a)} $R+\Delta f= n;$\\
{\rm(b)} $R+|\nabla f|^{2}- f={\rm constant};$\\
{\rm(c)} $\Delta_{f} R_{i\overline{j}}= R_{i\overline{j}}-R_{i\overline{j}k\overline{l}}R_{l\overline{k}};$\\
{\rm(d)} $\Delta_{f} R= R-|{\rm Ric}|^{2}.$\\
Here $|\nabla f|^{2}=g^{i\overline{j}}\nabla_{i}f\nabla_{\overline{j}}f$, $\Delta_{f}R_{i\overline{j}}=\Delta R_{i\overline{j}}-g^{k\overline{l}}\nabla_{k}f\nabla_{\overline{l}}R_{i\overline{j}}.$
\end{lemma}
\begin{proof} Both (a) and (b) follow from the soliton equation (\ref{solitonequation}) and the Bochner formula. For the convenience of the reader, the proof of (c) is given
in the appendix. Taking trace over (c), we get (d).
\end{proof}

\begin{theorem}\label{compactnessresult}Suppose $(M^n, g,  f)$ is a complete  shrinking gradient K\"ahler-Ricci soliton , if we assume the bisectional curvature is nonnegative and the Ricci curvature is positive, then it must be compact.
\end{theorem}
\proof Assume that $(M, g, f)$ is noncompact.  From Lemma \ref{several Identity} , we know that $R+|\nabla f|^2-f=C$. By adding a constant to $f$ if necessary, we can assume that $R+|\nabla f|^2=f$.

Concerning the potential $f$,
Cao  and Zhou \cite{Cao-Zhou} proved that
\begin{eqnarray}\label{quadratic growth}
\frac{1}{2}(d(x,p)-C_1)^2\leq f(x)\leq \frac{1}{2}(d(x,p)+C_2)^2,
\end{eqnarray}
where $p$ is a fixed point, $C_1$ and $C_2$ are positive constants depending only on the dimension of the manifold and the geometry of the unit ball $B_p(1)$.

Denote $\lambda(x)$ as the minimal eigenvalue of the Ricci curvature at $x$, suppose $v$ is the eigenvector corresponding to $\lambda(x)$, then
\begin{eqnarray*}
R_{i\overline{j}k\overline{l}}R_{\overline{k}l}v^i v^{\overline{j}}= R(v,\overline{v},\frac{\partial}{\partial z^{k}},\frac{\partial}{\partial z^{\overline{l}}})R_{\overline{k}l}
\end{eqnarray*}
Diagonalizing $Ric$ at $x$ so that $R_{\overline{k}l}=\lambda_k \delta_{kl}$. Since the bisectional curvature is nonnegative, we have
\begin{eqnarray}\label{nonnegativity}
R_{i\overline{j}k\overline{l}}R_{\overline{k}l}v^i v^{\overline{j}}=R(v,\overline{v},\frac{\partial}{\partial z^{k}},\frac{\partial}{\partial z^{\overline{k}}})\lambda_k \geq 0
\end{eqnarray}
Hence $\lambda$
satisfies the following differential inequality in the sense of barrier,
\begin{equation}
\Delta_f \lambda\leq \lambda.
\end{equation}
Actually, this means that at any point $x$, we can find a smooth function $h$ such that $h(x)=\lambda(x)$, $h\geq \lambda$ on
 $B(x, \delta)$ and $\Delta_f h(x)\leq h(x)$, where $\delta$ is a small positive constant, $h$ can be constructed as follows.

 Suppose $v\in T_x^{1,0}M$ is the unit eigenvector corresponding to $\lambda(x)$,
  taking parallel translation of $v$  along any unit speed geodesic starting from $x$, then in a small neighborhood $B(x, \delta)$, we get a
  smooth vector field $V$ with $V(x)=v$. Define $h(y)=Ric(y)(V(y), \overline{V(y)})$, then $h\geq \lambda$ for $y\in B(x,\delta)$ and $h(x)=\lambda(x)$, moreover, using $\nabla V(x)=0$ and (\ref{nonnegativity}), we get
\begin{eqnarray*}
\Delta_f h(x)=(\Delta_f Ric) (V(x),\overline{ V(x)})=\Delta_f R_{i\overline{j}}v^i v^{\overline{j}}=(-R_{i\overline{j}k\overline{l}}R_{\overline{k}l}+R_{i\overline{j}}(x))v^i v^{\overline{j}}\leq Ric(x)(v,\overline{v})=h(x).
\end{eqnarray*}

Choose a geodesic ball $B(p,r)$ of radius $r$ large enough, because the Ricci curvature is positive, then
\begin{eqnarray*}
a=\min_{\partial B(p, r)}\lambda>0.
\end{eqnarray*}
 We define
\begin{eqnarray*}
U=\lambda-\frac{a}{f}-\frac{2n a}{f^2}.
\end{eqnarray*}
From the growth rate of $f$, it follows that if $r$ is large enough , $U>0$ on $\partial B(p,r)$.

\begin{equation*}
\Delta_f U=\Delta_f \lambda-\Delta_f \left(\frac{a}{f}\right)-\Delta_f \left(\frac{2n a}{f^2}\right)\leq \lambda-a(\frac{1}{f}-\frac{n}{f^2})-2 n a \cdot \frac{3}{2 f^2}=U.
\end{equation*}
We have now proved $\Delta_f U\leq U$ on $M\backslash B(p,r)$ if $r$ is large enough.

Next we prove $U\geq 0$ on $M\backslash B(p,r)$. If there is a point $y_0\in M\backslash B(p,r)$ such that $U(y_0)<0$, then there must exist
a point $x_0\in M\backslash B(p,r)$ such that
$U(x_0)=\min_{y\in \{M\backslash B(p,r)\}}U(y)<0$, this is due to $\underline{\lim}_{d(x,p)\rightarrow \infty}U(x)\geq 0$ and $U>0$ on $\partial B(p,r)$.

At $x_0$, suppose $v$ is the unit eigenvector corresponding to $\lambda(x_0)$. Taking parallel translation along all the unit speed geodesic starting from $x_0$,
then in a small ball $B(x_0, \delta)$ we get a smooth vector field $V$ with $V(x_0)=v$. Define $\widetilde{U}=Ric(V(y), V(y))-\frac{a}{f}-\frac{2 n a}{f^2}$,
 then for any $y\in B(x_0, \delta)$, $\widetilde{U}(y)\geq U(y)$ and $\widetilde{U}(x_0)=U(x_0)$, so
\begin{eqnarray*}
0\leq \Delta_f \widetilde{U}(x_0)\leq \widetilde{U}(x_0)<0.
\end{eqnarray*}
Contradiction, so $U\geq 0$ on $M\backslash B(p, r)$, i.e. $Ric\geq \frac{a}{f}$.
As the argument in \cite{Munteanu-Wang}, we get $R\geq b \cdot \log f$ for some $b>0$.

 From the soliton equation (\ref{solitonequation}),
we get $\Delta f+R=n$. Consider the sublevel set $\{f\leq c\}$ of $f$, integrate, we get
\begin{eqnarray*}
\int_{\{f\leq c\}}n&=&\int_{\{f\leq c\}}(\Delta f+ R)=\int_{\{f=c\}}|\nabla f|+\int_{\{f\leq c\}}R\geq \int_{\{f\leq c\}}R.
\end{eqnarray*}
So the average of $R$ over $\{f\leq c\}$ is less than $n$, then it is easy to see that the average of $R$ over $B(p,r)$ is less than some $A>n$.

On the other hand, using the argument in \cite{Munteanu-Wang}, picking $q$ with $d(p,q)=\frac{3 r}{4}$, by (\ref{quadratic growth}),  we have
\begin{eqnarray*}
\int_{B(p,r)} R \geq \int_{B(q,\frac{r}{4})}R \geq b\cdot\log\left(\frac{1}{2}(\frac{r}{2}-C_1)^2\right){\rm Vol}\left(B(q,\frac{r}{4})\right).
\end{eqnarray*}
Applying Bishop-Gromov relative volume comparison, we get
\begin{eqnarray*}
{\rm Vol}\left(B(q,\frac{r}{4})\right)\geq c(n){\rm Vol}\left(B(q,2r)\right)\geq c(n) {\rm Vol}\left(B(p,r)\right).
\end{eqnarray*}
Hence
\begin{eqnarray*}
\int_{B(p,r)}R \geq  b\cdot c(n)\cdot\log\left(\frac{1}{2}(\frac{r}{2}-C_1)^2\right){\rm Vol}\left(B(p,r)\right),
\end{eqnarray*}
this means that the average of $R$ over $B(p,r)$ is greater that $A>n$ if $r$ is sufficiently large. Contradiction.
\qed

\section*{Appendix}
\noindent In this appendix we give a proof of (c) in Lemma (\ref{several Identity}). It is easy to obtain this identity using Ricci flow, for the evolution of Ricci curvature, refer (Cor. 2.83 in \cite{Chow et}).
Using the equation of K\"ahler-Ricci soliton (\ref{solitonequation}) , we have
\begin{eqnarray*}
\begin{aligned}
\quad\Delta_{f}R_{i\overline{j}}&=-\Delta \nabla_{i}\nabla_{\overline{j}}f-\nabla_{k}f\nabla_{\overline{k}}R_{i\overline{j}}\\
&=-\nabla_{\overline{j}}\nabla_{i}\Delta f+R_{i\overline{j}k\overline{l}}\nabla_{l}\nabla_{\overline{k}}f-\frac{1}{2}R_{i\overline{k}}\nabla_{k}\nabla_{\overline{j}}f-\frac{1}{2}R_{k\overline{j}}\nabla_{i}\nabla_{\overline{k}}f-\nabla_{k}f\nabla_{\overline{k}}R_{i\overline{j}}\\
&=\nabla_{\overline{j}}(R_{i\overline{k}}\nabla_{k}f)-\nabla_{k}f\nabla_{\overline{k}}R_{i\overline{j}}+R_{i\overline{j}k\overline{l}}\nabla_{l}\nabla_{\overline{k}}f-\frac{1}{2}R_{i\overline{k}}\nabla_{k}\nabla_{\overline{j}}f-\frac{1}{2}R_{k\overline{j}}\nabla_{i}\nabla_{\overline{k}}f\\
&=\nabla_{\overline{j}}R_{i\overline{k}}\nabla_{k}f+R_{i\overline{k}}\nabla_{k}\nabla_{\overline{j}}f-\nabla_{k}f\nabla_{\overline{k}}R_{i\overline{j}}+R_{i\overline{j}k\overline{l}}\nabla_{l}\nabla_{\overline{k}}f-\frac{1}{2}R_{i\overline{k}}\nabla_{k}\nabla_{\overline{j}}f-\frac{1}{2}R_{k\overline{j}}\nabla_{i}\nabla_{\overline{k}}f\\
&=R_{i\overline{j}k\overline{l}}\nabla_{l}\nabla_{\overline{k}}f+\frac{1}{2}R_{i\overline{k}}\nabla_{k}\nabla_{\overline{j}}f-\frac{1}{2}R_{k\overline{j}}\nabla_{i}\nabla_{\overline{k}}f\\
&=R_{i\overline{j}k\overline{l}}( g_{l\overline{k}}-R_{l\overline{k}})+\frac{1}{2}R_{i\overline{k}}( g_{k\overline{j}}-R_{k\overline{j}})-\frac{1}{2}R_{k\overline{j}}( g_{i\overline{k}}-R_{i\overline{k}})= R_{i\overline{j}}-R_{i\overline{j}k\overline{l}}R_{l\overline{k}}\\
\end{aligned}
\end{eqnarray*}
where we have used $\nabla_{\overline{j}}R_{i\overline{k}}=\nabla_{\overline{k}}R_{i\overline{j}}$ in the fifth equality. Therefore we obtain formula (c).

\section*{Acknowledgements}
The second author would like to thank Math Department of Rutgers University for its hospitality. He is partially supported by NSFC No. 11301017, Research Fund for the Doctoral Program of Higher Education of China, the Fundamental Research Funds for the Central Universities and the Scholarship from China Scholarship Council. We are grateful to the referee for some helpful suggestions.

\end{document}